\documentclass[12pt,a4paper]{amsart}
\usepackage{amssymb,amsthm,amsmath,mathtools,mathabx}
\usepackage{tikz,tikz-cd}
\usepackage{graphicx,color}
\usepackage{rotating}
\usepackage{hyperref}
\newtheorem{theorem}{Theorem} 

\newtheorem{proposition}[theorem]{Proposition}
\theoremstyle{definition}

\theoremstyle{remark}

\title{The number of Prime Parking Functions}
\author[R.~Duarte]{Rui Duarte}
\address{CIDMA and Department of Mathematics, University of Aveiro, 3810-193 Aveiro, Portugal}
\email{rduarte@ua.pt}

\author[A.~Guedes~de~Oliveira]{Ant\'onio Guedes de Oliveira}
\address{CMUP and Department of Mathematics, Faculty of Sciences, University of Porto, 4169-007 Porto, Portugal}
\email{agoliv@fc.up.pt}

\thanks{The authors were partially supported by CIDMA and CMUP, respectively, which are financed by national funds through Funda\c{c}\~ao para a Ci\^encia e a Tecnologia (FCT) within projects UIDB/04106/2020 (CIDMA) and UIDB/00144/2020 (CMUP)}

\newcommand\R{\mathbb{R}}
\newcommand\Z{\mathbb{Z}}
\newcommand\N{\mathbb{N}}
\newcommand\ba{\mathbf{a}}
\newcommand\bb{\mathbf{b}}
\newcommand\bc{\mathbf{c}}

\newcommand\red{\color{red}}
\newcommand\pf[1]{\text{\small PF}_{#1}}
\newcommand\ppf[1]{\text{\small PF}'_{#1}}

\newcommand{\carroe}[1]{%
\tikz[baseline=(tocarroe.base)]{\useasboundingbox (-.1,-.175) rectangle (10pt,2pt);%
\node[inner sep=.5pt,outer sep=.5pt] (tocarroe) {};
 \draw[thick,double distance = 8pt,line cap=round] ([yshift=-1pt] tocarroe.north west)
 -- ([yshift=-1pt,xshift=8pt] tocarroe.north west) ;
 \draw[thick,double distance = 8pt,line cap=round] 
   ([yshift=-1pt] tocarroe.north west) -- ([yshift=-1pt,xshift=4pt] tocarroe.north west) ;
 \node (tocarroe) {\footnotesize$#1$};
   }}
\newcommand{\carrop}[1]{%
\tikz[baseline=(tocarrop.base)]{\useasboundingbox (-.1,-.175) rectangle (10pt,2pt);%
\node[inner sep=.5pt,outer sep=.5pt] (tocarrop) {};
 \draw[red,thick,double distance = 8pt,line cap=round] ([yshift=-1pt] tocarrop.north west)
 -- ([yshift=-1pt,xshift=8pt] tocarrop.north west) ;
 \draw[red,thick,double distance = 8pt,line cap=round] 
   ([yshift=-1pt] tocarrop.north west) -- ([yshift=-1pt,xshift=4pt] tocarrop.north west) ;
 \node (tocarrop) {\red\footnotesize$#1$};
\draw[red,thick,<-] ([xshift=15pt] tocarrop.south east) arc (0:90:7.5pt);
  }}
\newcommand{\carropd}[1]{%
\tikz[baseline=(tocarropd.base)]{\useasboundingbox (-.1,-.175) rectangle (10pt,2pt);%
\node[inner sep=.5pt,outer sep=.5pt] (tocarropd) {};
 \draw[red,thick,double distance = 8pt,line cap=round] ([yshift=-1pt] tocarropd.north west)
 -- ([yshift=-1pt,xshift=8pt] tocarropd.north west) ;
 \draw[red,thick,double distance = 8pt,line cap=round] 
   ([yshift=-1pt] tocarropd.north west) -- ([yshift=-1pt,xshift=4pt] tocarropd.north west) ;
 \node (tocarropd) {\red\footnotesize$#1$};
\draw[red,thick] ([xshift=5pt,yshift=7.5pt] tocarropd.south east) -- ([xshift=10pt,yshift=7.5pt] tocarropd.south east) ;
\draw[red,thick,<-] ([xshift=30pt,yshift=-12.5pt] tocarropd.south east) arc (0:90:20pt);
}}
\newcommand{\carropq}[4]{%
\tikz[baseline=(tocarropq.base)]{\useasboundingbox (-.1,-.175) rectangle (10pt,2pt);%
\node[inner sep=.5pt,outer sep=.5pt] (tocarropq) {};
 \draw[red,thick,double distance = 8pt,line cap=round] ([yshift=-1pt] tocarropq.north west)
 -- ([yshift=-1pt,xshift=8pt] tocarropq.north west) ;
 \draw[red,thick,double distance = 8pt,line cap=round] 
   ([yshift=-1pt] tocarropq.north west) -- ([yshift=-1pt,xshift=4pt] tocarropq.north west) ;
 \node (tocarropq) {\red\footnotesize$#4$};
\draw[red,thick] ([xshift=#1pt,yshift=7.5pt] tocarropq.south east) -- ([xshift=#2pt,yshift=7.5pt] tocarropq.south east) ;
\draw[red,thick,<-] ([xshift=#3pt] tocarropq.south east) arc (0:90:7.5pt);
}}
\newcommand{\carropr}[4]{%
\tikz[baseline=(tocarropr.base)]{\useasboundingbox (-.1,.175) rectangle (10pt,2pt);%
\node[inner sep=.5pt,outer sep=.5pt] (tocarropr) {};
 \draw[red,thick,double distance = 8pt,line cap=round] ([yshift=0pt] tocarropr.north west)
 -- ([yshift=0pt,xshift=8pt] tocarropr.north west) ;
 \draw[red,thick,double distance = 8pt,line cap=round] 
   ([yshift=0pt] tocarropr.north west) -- ([yshift=0pt,xshift=4pt] tocarropr.north west) ;
 \node (tocarropr) {\raisebox{2pt}{\red\footnotesize$#4$}};
\draw[red,thick] ([xshift=#1pt,yshift=10pt] tocarropr.south east) -- ([xshift=#2pt,yshift=10pt] tocarropr.south east) ;
\draw[red,thick,<-] ([xshift=#3pt] tocarropr.south east) arc (0:90:10pt);
}}

\begin{document}
\maketitle
\begin{abstract}
A parking function of length $n$ is prime if we obtain a parking function of length $n-1$ by deleting one 1 from it. In this note we give a new direct proof that the number of prime parking functions of length $n$ is $(n-1)^{n-1}$. This proof leads to a new interpretation, in close terms to the definition of parking function.
\end{abstract}

Let us suppose that $n$ cars enter some one-way street with $n$ parking places marked $1,2,\dotsc,n$, and each driver intends to park in a particular spot, say, driver $i$ intends to park in the spot marked $p(i)\in[n]$. 
Suppose that, when driver $i$ finds his favorite spot free, he parks there, but when he finds it occupied by any of the previous $i-1$ drivers he looks for the next free space after spot $p(i)$ for parking, and if none exists, then leaves the street.
For which functions $p\in[n]^n$ is it possible that all drivers park and do not leave the street? Let us call the suitable functions \emph{parking functions of length $n$}~\footnote{Parking functions have attracted much attention since its definition by Konheim and Weiss. For a recent survey, see~\cite{Yan}.}.

The answer to this question was given by Konheim and Weiss~\cite{KW} as follows: suppose that 
we take the nondecreasing rearrangement
~\footnote{That is, $q_i=p_{\pi_i}$ for some permutation
$\pi\in\mathfrak{S}_n$ such that $q_1\leq q_2\leq\dotsb\leq q_n$.}
$\mathbf{q}=(q_1,\dotsc,q_n):=\big(p(1),\dotsc,p(n)\big)^\uparrow$.
Then, the \emph{parking functions} are those $p\in[n]^n$ for which
\begin{align}
&q_i\leq i,\text{ for every $i\in[n]$\,.}\label{eq.pf1}
\shortintertext{In this article we are mainly interested in the subclass of \emph{prime parking functions}, which are those $p\in[n]^n$ such that~\footnotemark}
&q_i< i,\text{ for every $i\in[n]$ with $i>1$}\,.\label{eq.ppf1}
\end{align}
\footnotetext{By definition, we may (and will) consider a length $n$ prime parking function as a function $p\in[n-1]^n$.}

For an example of a parking function, suppose that $n=15$ and take
\[\ba\text{\footnotesize$:=(3, 13, 6, 3, 7, 3, 2, 1, 10, 11, 6, 7, 14, 10, 11)\in[15]^{15}$}\]
Then all cars park, in the previously described sense. In fact, the parking goes as follows.
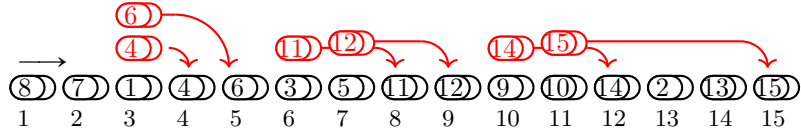
\begin{center}
\setlength{\tabcolsep}{4.25pt}
\begin{tabular}[h]{p{.4cm} p{.4cm} p{.4cm} p{.4cm} p{.4cm} p{.4cm} p{.4cm} p{.4cm} p{.4cm} p{.4cm} p{.4cm}
p{.4cm} p{.4cm} p{.4cm} p{.4cm}}
&&\carropd{6}\\[-1.5pt]
\raisebox{-5pt}{$\longrightarrow$}&&\carrop{4}&&&\carropq{2}{22.5}{30}{11}&\carropr{2}{20}{30}{12}&&&\carropq{2}{22.5}{30}{14}&
\carropr{2}{60}{70}{15}\\
\carroe{8}&\carroe{7}&\carroe{1}&\carroe{4}&\carroe{6}&\carroe{3}&\carroe{5}
&\carroe{11}&\carroe{12}&\carroe{9}&\carroe{10}&\carroe{14}&\carroe{2}&\carroe{13}&\carroe{15}\\[-2.5pt]
\text{\scriptsize1}&\text{\scriptsize2}&\text{\scriptsize3}&\text{\scriptsize4}&\text{\scriptsize5}
&\text{\scriptsize6}&\text{\scriptsize7}&\text{\scriptsize8}&\text{\scriptsize9}&\text{\scriptsize10}
&\text{\scriptsize11}&\text{\scriptsize12}&\text{\scriptsize13}&\text{\scriptsize14}&\text{\scriptsize15}
\end{tabular}
\end{center}
On the other hand,
{\footnotesize$\ba^\uparrow=(1, 2, 3, 3, 3, 6, 6, 7, 7, 10, 10, 11, 11, 13, 14)$}, which
is componentwise less or equal to $(1,2,\dotsc,15)$.

\noindent
\begin{minipage}[h]{.4\textwidth}
\setlength{\parindent}{1em}
\indent
Parking functions occur when we consider in $\R^n$ the set of hyperplanes defined by 
the equations $x_i =x_j + k$, where $1\leq i<j\leq n$ and $k=0,1$, all of which contain the line $\ell$
of equation $\ x_1=x_2=\dotsb=x_n\ $.
Hence, we may represent each hyperplane as the hyperline given by its intersection with a given hyperplane orthogonal to
$\ell$.
\end{minipage}\hfill
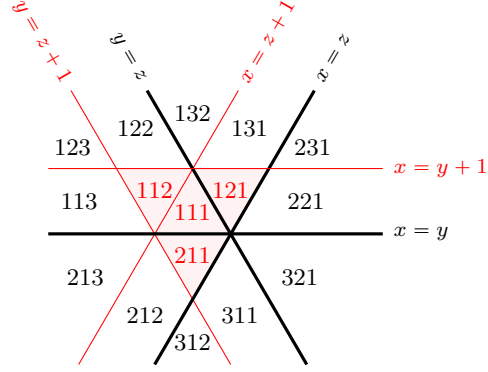
\begin{minipage}[t]{.5\textwidth}
\begin{tikzpicture}[scale=2,baseline=-1.75]
\draw[white,fill=red!5] (0.25, -0.433) -- (-0.25, 0.433) -- (0.75, 0.433) -- cycle;
\node[right] at (1.5,0.433) (l1) {\red\tiny$x=y+1$};
\node[right] at (1.5,0) (l1) {\tiny$x=y$};
\node[rotate=60,right] at (1.05,.95) (l1) {\tiny$x=z$};
\node[rotate=-60,left] at (-0.05,.95) (l1) {\tiny$y=z$};
\node[rotate=60,right] at (0.55,.95) (l1) {\red\tiny$x=z+1$};
\node[rotate=-60,left] at (-0.55,.95) (l1) {\red\tiny$y=z+1$};
\draw[thin,red] (0.5,-0.866) -- (-0.55,.95) ;
\draw[thin,red] (0.55,.95) -- (-0.5,-0.866) ;
\draw[very thick] (1.00,-0.866) -- (-0.05,.95) ;
\draw[very thick] (1.05,.95) -- (0,-0.866) ;
\draw[thin,red] (1.5,0.433) -- (-0.7,0.433) ;
\draw[very thick] (1.5,0) -- (-0.7,0) ;
\node at (0.996,0.217) (a) {\scriptsize$221$} ;
\node at (1.04,0.577) (a) {\scriptsize$231$} ;
\node at (0.625,0.7) (a) {\scriptsize$131$} ;
\node at (0.250,0.8) (a) {\scriptsize$132$} ;
\node at (-0.125,0.7) (a) {\scriptsize$122$} ;
\node at (-0.539,0.577) (a) {\scriptsize$123$} ;
\node at (-0.496,0.217) (a) {\scriptsize$113$} ;
\node at (0,0.289) (a) {\red\scriptsize$112$} ;
\node at (0.250,0.144) (a) {\red\scriptsize$111$} ;
\node at (0.500,0.289) (a) {\red\scriptsize$121$} ;
\node at (-0.455,-0.289) (a) {\scriptsize$213$} ;
\node at (-0.0625,-0.541) (a) {\scriptsize$212$} ;
\node at (0.250,-0.722) (a) {\scriptsize$312$} ;
\node at (0.250,-0.144) (a) {\red\scriptsize$211$} ;
\node at (0.562,-0.541) (a) {\scriptsize$311$} ;
\node at (0.955,-0.289) (a) {\scriptsize$321$} ;
\end{tikzpicture}
\end{minipage}
\smallskip

Label  the region $R_0$ defined by $\ x_1>x_2>\dotsb>x_n>x_1+1\ $ with $\ (1,1,\dotsc,1)\in\Z^n$.
Now, given regions $R$ and $R'$ separated by a unique hyperplane $H$ so that $R$ and $R_0$ are on the same side of $H$, give to $R'$ the label of $R$ but add $1$ to one coordinate,
the $i$.th coordinate if $H$ is defined by an equation of form $\ x_i=x_j\ $ ($i<j$), and 
the $j$.th coordinate if $H$ is defined by an equation of form $\ x_i=x_j+1\ $.

Pak and Stanley showed~\cite{Stand}  that the labels thus defined are exactly the parking functions of length $n$.
We proved~\cite{DGO}  that the {prime parking functions are the Pak-Stanley labels of the bounded regions} (shaded on the figure above, where the case $n=3$ is represented).

\section{Main theorem}

Let $\pf{n}$ and $\ppf{n}\subseteq\pf{n}$ be the sets of parking functions and of prime parking functions of length $n$, respectively.
We know that
\begin{align}
|\pf{n}|&=(n+1)^{n-1}\label{numpf}
\shortintertext{and}
|\ppf{n}|&=(n-1)^{n-1}\label{numppf} .
\end{align}
The goal of this paper is to show Equation~\ref{numppf} in a simple and original manner, as
a direct consequence of Theorem~\ref{mthm}, below.

Prime parking functions were introduced by Gessel, who proved Equation~\ref{numppf} through generating functions (Cf. \cite[Exercise 5.49 f]{Stan}). Soon thereafter, Kalikow~\cite[p.37]{Kal} gave a direct proof, which inspired ours.
\begin{theorem}\label{mthm}
For every $\ba=(a_1,\dotsc,a_n)\in[n-1]^n$ there is a unique $k\in[n-1]$ and a unique prime parking function $\bb=(b_1,\dotsc,b_n)$ such that
\[a_i\equiv b_i+k-1\pmod{n-1}\quad\text{ for every $i\in[n]$}\,.\]
\end{theorem}
\begin{proof} 
Without loss of generality, suppose that $\ba^\uparrow=\ba$, and define,
for every $i\in[n-1]$,
\begin{align*}
s_i&=\text{\footnotesize $(a_{i+1}-a_i)+\dotsb+(a_n-a_i)+(a_1-a_i+n-1)+\dotsb+(a_{i-1}-a_i+n-1)$}\\
&=\sum_{j=1}^n a_j-n\,a_i+(i-1)(n-1)\,,
\shortintertext{and}
s_n&=(a_1-a_n+n-1)+\dotsb+(a_{n-1}-a_n+n-1)\\
&=\sum_{j=1}^n a_j-n\,a_n+(n-1)^2\,,
\end{align*}
and let $d\in[n]$ be such that $s_d=\min\{s_i\mid i\in[n]\}$. Let $k=a_d$, and
\begin{align*}
&\bb=\text{\footnotesize $(1,a_{d+1}-k+1,\dotsc,a_n-k+1,a_1-k+n,\dotsc,a_{d-1}-k+n)$}\,.
\shortintertext{Since, for $i\neq d$,}
&s_i-s_d=(n-1)(i-d)-n(a_i-a_d)\geq0\,,\\
&a_i-k
\leq(i-d)-
\frac{i-d}{n} <
\begin{cases}i-d,&\text{if $i > d$;}\\
i-d+1,&\text{if $i< d$.}\end{cases}
\end{align*}
and $\bb$ is a prime parking function for it is componentwise less or equal to
\[(1,1,\dotsc,n-d+1\,,\,n-d+2,\dotsc,n-1)\,.\]

For proving the uniqueness, now, suppose, contrary to our assumptions, that both $\bb=(b_1,\dotsc,b_n)$ and $\bc=(c_1,\dotsc,c_n)$ are prime parking functions and verify the hypothesis of our theorem. Hence, there exist $\ell\in\N$ such that $c_i\equiv b_i+\ell\pmod{n-1}$ for every $i\in[n]$.
Without loss of generality, again, we suppose that $\bb^\uparrow=\bb$.
If $b_n+\ell<n$, then $c_1=1+\ell=1$, and $\ell=0$ and $\bb=\bc$. Otherwise,
let $b_{i-1}<n-\ell\leq b_i$.
Then
\[\bc^\uparrow=(\overbrace{\text{\footnotesize$b_i+\ell-n+1$}}^{1},\dotsc,
\overbrace{\text{\footnotesize$b_n+\ell-n+1$}}^{{}\leq n-i},
\overbrace{\text{\footnotesize$1+\ell$}}^{{}\leq n-i+1},\dotsc,
\overbrace{\text{\footnotesize$b_{i-1}+\ell$}}^{{}\leq n-1})\,.\]
Hence, $b_i=n-\ell<i$ and $\ell+1\leq n-i+1$.
This is a contradiction, and it concludes the proof.
\end{proof}
 \noindent
For example, as before, let
\begin{align*}
&\ba=\text{\small$(3, 13, 6, 3, 7, 3, 2, 1, 10, 11, 6, 7, 14, 10, 11)\in\pf{15}$}\,.
\shortintertext{Then, if $k=10$ and
$a_i\equiv b_i+k-1\pmod{n-1}$ for every $i\in[n]$,}
&\bb=\text{\small$(8, 4, 11, 8, 12, 8, 7, 6, 1, 2, 11, 12, 5, 1, 2)\in\ppf{15}$}\,.
\end{align*}

\section{The parking point of view}
Note that, by \eqref{eq.pf1} and \eqref{eq.ppf1}, a prime parking function of length $n$ is a parking function of length $n-1$ with an extra $1$. More precisely,
if $\ba=(a_1,\dotsc,a_n)\in[n-1]^n$, $a_i=1$ for some $i\in[n]$, and
$\bb=(a_1,\dotsc,a_{i-1},a_{i+1},\dotsc,a_n)\in[n-1]^{n-1}$, then $\ba\in\ppf{n}$ if and only if $\bb\in\pf{n-1}$.
Hence, $\ba\in[n-1]^n$ is a prime parking function if all cars can park in a one-way street with $n$ spots labeled $\ {\red1},{\red1},2,\dots,n-1$. 
For instance, 
\[\text{\small$\bb=(8, 4, 11, 8, 12, 8, 7, 6, 1, 2, 11, 12, 5, 1, 2)\in\ppf{15}$}\]
parks as follows.
\begin{center}
\setlength{\tabcolsep}{4.25pt}
\begin{tabular}[h]{p{.4cm} p{.4cm} p{.4cm} p{.4cm} p{.4cm} p{.4cm} p{.4cm} p{.4cm} p{.4cm} p{.4cm} p{.4cm}
p{.4cm} p{.4cm} p{.4cm} p{.4cm}}
$\longrightarrow$\\[-5pt]
\carroe{9}&\carroe{14}&\carroe{10}&\carroe{15}&\carroe{2}&\carroe{13}&\carroe{8}
&\carroe{7}&\carroe{1}&\carroe{4}&\carroe{6}&\carroe{3}&\carroe{5}&\carroe{11}&\carroe{12}\\[-2.5pt]
\text{\scriptsize\red1}&\text{\scriptsize\red1}&\text{\scriptsize2}&\text{\scriptsize3}&\text{\scriptsize4}&\text{\scriptsize5}
&\text{\scriptsize6}&\text{\scriptsize7}&\text{\scriptsize8}&\text{\scriptsize9}&\text{\scriptsize10}
&\text{\scriptsize11}&\text{\scriptsize12}&\text{\scriptsize13}&\text{\scriptsize14}
\end{tabular}
\end{center}

In fact,
if all elements of $\ba$ park as described, the element of $\ba$ that parks in the first spot labeled $1$ must be 
$a_i=1$ so that $a_j>1$ for every $j<i$. Then, let $\bb$ be as defined above. Since it is parked in the remaining places, $\bb$ is a parking function.
On the other hand, if, again, $a_i=1$ for some $i\in[n]$ such that $a_j>1$ for every $j\in[i-1]$, 
then the parking procedure parks $a_i$ in the first spot labeled $1$, and parks the remaining elements of $\ba$ in the next positions, since they are the elements of $\bb$.

An immediate consequence of Theorem~\ref{mthm} is the following addendum to Konheim and Weiss' result.
\begin{proposition}
Let $n$ cars enter a one-way \emph{street} with $n$ spots, and suppose driver $i$ wants to park in position
$p(i)\leq n-1$.
Then, there is a unique $k\in[n-1]$ such that all cars can park in the previous sense, provided the spots are labeled consecutively
\[ {\red k},{\red k},k+1,\dots,n-1,1,2,\dotsc k-1\,.\]
\end{proposition}
For example, for  
{\footnotesize$\ba=(3, 13, 6, 3, 7, 3, 2, 1, 10, 11, 6, 7, 14, 10, 11)$}, we have
\begin{center}
\setlength{\tabcolsep}{4.25pt}
\begin{tabular}[h]{p{.4cm} p{.4cm} p{.4cm} p{.4cm} p{.4cm} p{.4cm} p{.4cm} p{.4cm} p{.4cm} p{.4cm} p{.4cm}
p{.4cm} p{.4cm} p{.4cm} p{.4cm}}
$\longrightarrow$\\[-5pt]
\carroe{9}&\carroe{14}&\carroe{10}&\carroe{15}&\carroe{2}&\carroe{13}&\carroe{8}
&\carroe{7}&\carroe{1}&\carroe{4}&\carroe{6}&\carroe{3}&\carroe{5}&\carroe{11}&\carroe{12}\\[-2.5pt]
\text{\scriptsize\red10}&\text{\scriptsize\red10}&\text{\scriptsize11}&\text{\scriptsize12}&\text{\scriptsize13}&\text{\scriptsize14}
&\text{\scriptsize1}&\text{\scriptsize2}&\text{\scriptsize3}&\text{\scriptsize4}
&\text{\scriptsize5}&\text{\scriptsize6}&\text{\scriptsize7}&\text{\scriptsize8}&\text{\scriptsize9}
\end{tabular}
\end{center}

\end{document}